\documentclass[11pt]{article}
\usepackage{amssymb}
\usepackage{amsthm}
\usepackage{amsfonts}
\usepackage{amsmath}
\usepackage{amsopn}
\usepackage{enumerate}

\def\sqr#1#2{{\vcenter{\vbox{\hrule height.#2pt
              \hbox{\vrule width.#2pt height#1pt \kern#1pt \vrule width.#2pt}
              \hrule height.#2pt}}}}
\def\3n{\negthinspace \negthinspace \negthinspace }
\def\2n{\negthinspace \negthinspace }
\def\1n{\negthinspace }

\def\={\buildrel \triangle \over =}

\def\({\Big (}
\def\){\Big )}
\def\[{\Big[}
\def\]{\Big]}
\def\bde{\begin{definition}}
\def\ede{\end{definition}}
\def\be{\begin{equation}}
\def\bel{\begin{equation}\label}
\def\ee{\end{equation}}
\def\bt{\begin{theorem}}
\def\et{\end{theorem}}
\def\bc{\begin{corollary}}
\def\ec{\end{corollary}}
\def\bl{\begin{lemma}}
\def\el{\end{lemma}}
\def\bp{\begin{proposition}}
\def\ep{\end{proposition}}
\def\bas{\begin{assumption}}
\def\eas{\end{assumption}}
\def\br{\begin{remark}}
\def\er{\end{remark}}
\def\ba{\begin{array}}
\def\ea{\end{array}}
\def\ed{\end{document}}

\def\square#1{\vbox{\hrule\hbox{\vrule height#1%
     \kern#1\vrule}\hrule}}
\def\rectangle#1#2{\vbox{\hrule\hbox{\vrule height#1%
     \kern#2\vrule}\hrule}}

\font\tenbb=msbm10 \font\sevenbb=msbm7 \font\fivebb=msbm5

\newfam\bbfam
\scriptscriptfont\bbfam=\fivebb \textfont\bbfam=\tenbb
\scriptfont\bbfam=\sevenbb

\newtheorem{lemma}{Lemma}[section]
\newtheorem{theorem}[lemma]{Theorem}
\newtheorem{remark}[lemma]{Remark}
\newtheorem{example}[lemma]{Example}
\newtheorem{corollary}[lemma]{Corollary}

\newtheorem{definition}[lemma]{Definition}
\newtheorem{proposition}[lemma]{Proposition}
\newtheorem{assumption}[lemma]{Assumption}

\makeatletter
   
   \@addtoreset{equation}{section}
\makeatother

\begin{document}
\title{Characterization by detectability inequality for periodic  stabilization of linear time-periodic evolution systems}

\author{
Yashan Xu\thanks{School of Mathematical Sciences, Fudan University, KLMNS, Shanghai, 200433,
China  (\texttt{yashanxu@fudan.edu.cn}). This work is
supported in part by NNSF Grant  11871166.}}

\date{}

\maketitle

\begin{abstract}
Given a linear time-periodic control system in a Hilbert space with a bounded control operator, we present a characterization of periodic stabilization in terms of a detectability inequality. Similar  characterization
was built up in \cite{Trelat2020} for time-invariant systems.
 \end{abstract}


\medskip


\section{Introduction}
\paragraph{Control system.}
 Let  $Y$ be a real Hilbert space (state space) with the norm and the inner product
$\Vert\cdot\Vert$ and $\langle\cdot,\cdot\rangle$ respectively. Let $U$ be
another real Hilbert space (control space) with the norm and the inner product
$\Vert\cdot\Vert_U$ and $\langle\cdot,\cdot\rangle_U$ respectively.
We identify  $Y$ (resp., $U$) with its dual $Y'$ (resp., $U'$).
Let $T>0$ be arbitrarily given.
In this paper, we will study the periodic stabilization for  the linear control system:
\begin{equation}\label{contsyst}
y'(t) = A(t)y(t)+B(t)u(t),\qquad t\in \mathbb{R}^+\triangleq[0,\infty),
\end{equation}
under  the following  hypotheses:
\begin{enumerate}[($H_1$)]
\item
The family of operators $\{A(t)\}_{t\geq 0}$ satisfies that
$A(t)=A+D(t)$ for a.e. $t\in (0,\infty)$, where the operator $A$, with its domain  $D(A)\subset Y$, generates a  $C_0$-semigroup $\{S(t)\}_{t\geq 0}$ on  $Y$; and the operator-valued function $D(\cdot)\in L^1_{loc}(0,\infty;{\cal L}(Y))$ is $T$-periodic in time, i.e., $D(t+T)=D(t)$ for a.e. $t\in \mathbb{R}^+$.
\item  The operator-valued function $B(\cdot)\in L^\infty(0,\infty;{\cal L}(U,Y))$ is $T$-periodic.
(We denote its norm by $\|B\|_{L^\infty}$.)
\item Each control  $u$ is taken from the space $L^2(0,\infty; U)$.
 \end{enumerate}
Given $u\in L^2(t,\infty;U)$, $z\in Y$ and $t\geq 0$, we write  $y(\cdot;t,z,u)$
 for the solution of the equation (\ref{contsyst}) over $[t,\infty)$ with the initial condition: $y(t)=z$.
 (When $z\in Y$ and $u\in L^2(0,\hat T;U)$ for some $\hat T>0$, we still use $y(\cdot;0,z,u)$ to denote
  the solution
  $y(\cdot;0,z,\hat u)$, where $\hat u=u$ over $(0,\hat T)$ and $\hat u=0$ over $(\hat T,\infty)$.)
Let
\begin{equation*}
E\triangleq\left\{(t,s)\in\mathbb{R}^+\times\mathbb{R}^+\bigm|0\le s\le t<\infty\right\}.
\end{equation*}
Let $\Phi(\cdot,\cdot): E\rightarrow {\cal L}(Y)$ be the evolution system generated by $A(\cdot)$.
(When $t\geq s\geq 0$, we denote by $\|\Phi(t,s)\|_{\mathcal{L}(Y)}$ the operator norm of $\Phi(t,s)$.)
 Then we have that (see   Lemma 5.6  on Page 68 in  \cite{Li}) that $\Phi(\cdot,\cdot)$ is strongly continuous over $E$,
and that
\begin{equation}\label{1.02}
\Phi(t,s)z=S(t-s)z+\displaystyle\int_s^tS(t-r)B(r)\Phi(r,s)z\mathrm dr,\;\;\mbox{when}\;\; 0\leq s\leq t<\infty\;\;\mbox{and}\;\; z\in Y.
\end{equation}
Moreover, it follows by  $(H_1)$-$(H_2)$ and (\ref{1.02}) that
\begin{equation}\label{1.05}
\Phi(t+T, s+T)=\Phi(t,s)\;\; \text{ for any }\;0\le s\le t<\infty.
\end{equation}

\medskip

\paragraph{Concepts on the stabilization.}
Several  concepts related to the periodic stabilization of the system (\ref{contsyst})  are given in order.

\begin{itemize}
\item The system (\ref{contsyst}) is said to be periodically exponentially stabilizable (periodically stabilizable, for short), if there is
$K(\cdot)\in L^\infty(\mathbb{R}^+;{\cal L}(Y,U))$, with $K(T+t)=K(t)$ for a.e. $t\in \mathbb{R}^+$,  so that
the system $y'(t)=\left[A(t)+B(t)K(t)\right]y(t)$ ($t\geq0$) is stable, i.e., for some  $M>0$ and $\omega>0$,
\begin{equation}\label{1.06}
\|\Phi_K(t,s) z\|\leq M e^{-\omega (t-s)}\|z\|\qquad\mbox{for any } z\in Y\;\mbox{and}\;t\geq s\geq0.
\end{equation}
Here and thereafter,
$\Phi_K(\cdot,\cdot): E\rightarrow {\cal L}(Y)$  denotes the evolution system generated by $A_K(\cdot) \triangleq A(\cdot) +B(\cdot)K(\cdot)$.

\item Given $n\in \mathbb{N}^+\triangleq\{1,2,\dots\}$, the following system is called  an  adjoint equation   of  \eqref{contsyst} over $[0,nT]$:
\begin{equation}\label{adjoint-equation}
\begin{cases}
\dot\varphi_n(t)=-A^*(t)\varphi_n(t) \quad\text{in }[0, nT],\\
\varphi_n(nT)=\psi.
\end{cases}
\end{equation}
where $\psi\in Y$. {\it We write $\varphi_n(\cdot; \psi)$ for the solution to the system \eqref{adjoint-equation}.}
\end{itemize}

\paragraph{Main result.} The main result of this paper is to present a characterization by an inequality
for the periodic stabilization of the system
 \eqref{contsyst}.

\begin{theorem}\label{maintheorem1}
Suppose that $(H_1)$-$(H_3)$ are true. Then the following statements are equivalent:
\begin{enumerate}[($ {E}_1$)]
\item The  system (\ref{contsyst}) is periodically stabilizable.
\item For any $ \delta\in(0,1)$, there is $n_\delta\in\mathbb{N}^+$ and $ C_{\delta}>0$ so that for any  $k\in\mathbb{N}^+$,
\begin{equation}\label{1.10}
\|\varphi_{kn_\delta}(0;\psi)\|\leq\delta^k\|\psi\|
+C_{\delta}\|B(\cdot)^*\varphi_{kn_\delta}(\cdot;\psi)\|_{L^2(0,kn_\delta T;U)},\;\;\mbox{when}\;\;\psi\in Y.
\end{equation}
\item There is $ \delta\in(0,1)$,  $n\in\mathbb{N}^+$  and $C>0$ so that
\begin{equation}\label{1.11}
\|\varphi_n(0;\psi)\|\le\delta\|\psi\|
+C\|B(\cdot)^*\varphi_n(\cdot;\psi)\|_{L^2(0,nT;U)},\;\;\mbox{when}\;\;\psi\in Y.
\end{equation}
\end{enumerate}
\end{theorem}
Several notes on Theorem \ref{maintheorem1} are given in order.
\begin{itemize}
\item
The similar equivalence results in Theorem \ref{maintheorem1} were obtained in \cite{Trelat2020} for  time-invariant systems. Different kinds of   characterizations of the periodic stabilization for time-periodic systems have been studied in \cite{Lundaridi}, \cite{Prato},
         \cite{WangXu} and \cite{WangXu-book}. The  characterization (for the system \eqref{contsyst}), given in Theorem \ref{maintheorem1}, seems to be new.
    \item We prefer to call  \eqref{1.10}  (or \eqref{1.11})  a detectability inequality rather than
    a weak observability inequality (which was used in  \cite{Trelat2020}). The reason will be given in
    Subsection \ref{subsection3.1}.
    \item It is well known that the null controllability of \eqref{contsyst} is equivalent to the standard observability inequality; the null controllability of \eqref{contsyst} implies the periodic stabilization of
        \eqref{contsyst}. (For the later, we refer readers to  \cite{WangXu-book}.)
Comparing the standard observability inequality and the detectability inequality \eqref{1.11} (or \eqref{1.10}), we see that the gap between the null controllability and the periodic stabilization
    can be quantified by adding the term $\delta\|\psi\|$ on the right hand side of the standard observability inequality.

\end{itemize}

The rest of the paper is organized as follows: Section 2 proves Theorem \ref{maintheorem1}.
 Section 3 gives some further discussions on the periodic stabilization.

\section{Proof of Theorem \ref{maintheorem1}}
Several lemmas  will be used in the proof of Theorem \ref{maintheorem1}.
The first one, i.e., Lemma \ref{lemma1},  is a direct consequence of  Theorem 1.4  in \cite{WangXu-book}.
To state it, we define, for each  $z\in Y$, the LQ problem:
\begin{equation}\label{1.15}
(LQ)_{z}\hspace{110pt} W(z)\triangleq \inf_{u\in L^2(0,\infty;U)}J(u;z),\hspace{110pt}\mbox{}
\end{equation}
where
 \begin{equation}\label{1.13}
J(u;z)\triangleq\int^{\infty}_0\Bigr[\|y(s; 0,z,u)\|^2+ \|u(s)\|^2_U\Bigl]\mathrm{d}s,
\;\; u\in L^2(0,\infty;U).
\end{equation}
\begin{lemma}\label{lemma1}
The following assertions  are equivalent:
\begin{enumerate}[(i)]
\item The system (\ref{contsyst}) is  periodically  stabilizable.
\item  The functional $W(\cdot)$, given by \eqref{1.15}, is finite valued, i.e.,   $W(z)<\infty$ for each $z\in Y$.
\end{enumerate}
\end{lemma}

The second one, i.e.,  Lemma \ref{lemma2}, takes some ideas from  Proposition 6 in \cite{Trelat2020}.

\begin{lemma} \label{lemma2}
Let $ \delta\in(0,1)$, $n\in\mathbb{N}^+$, and $ C>0$. Then the following statements are equivalent:
\begin{enumerate}[(i)]
\item The system (\ref{contsyst}) has the property: for any $z\in Y$, there is $u_z\in L^2(0,nT;U)$ so that
\begin{equation}\label{estimate-state-control}
\|y(n T;0,z,u_z)\|\leq\delta\|z\|\qquad\text{and}\qquad
\|u_z\|_{ L^2(0,n T;U)}\leq C\|z\|.
\end{equation}
\item The adjoint equation \eqref{adjoint-equation} has the property:
\begin{equation}\label{1.17}
\|\varphi_n(0;\psi)\|\leq\delta\|\psi\|
+C\|B(\cdot)^*\varphi_n(\cdot;\psi)\|_{L^2(0,nT;U)}\qquad {\mbox{for any }}\psi\in Y.
\end{equation}
\end{enumerate}

\end{lemma}

\begin{proof}
$(i)\Longrightarrow(ii)$. Arbitrarily fix  $\psi\in Y$ and $z\in Y$. Let $u_z\in L^2(0,nT;U)$
be given by $(i)$. Then by (\ref{contsyst}) and \eqref{adjoint-equation}, we have
that for any $z\in Y$,
$$
\langle y(nT;0,z,u_z),\psi\rangle-\langle z,\varphi_n(0;\psi)\rangle=\int^{nT}_0\langle u_z(t),\, B(t)^*\varphi_n(t;\psi)\rangle_U\mathrm dt,
$$
where  $u_z\in L^2(0,nT;U)$
is given by $(i)$.
From the above, it follows  that
\begin{eqnarray*}
\|\varphi_n(0;\psi)\|
&=&
\sup\limits_{z\neq 0}\frac{\langle z,\varphi_n(0;\psi)\rangle}{\|z\|}\\
&=&
\sup\limits_{z\neq 0}\frac{ \langle y(nT;0,z,u_z),\psi\rangle-\displaystyle\int^{nT}_0\langle u_z(t),\, B(t)^*\varphi_n(t;\psi)\rangle_U\mathrm dt}{\|z\|}\\
&\leq&
\sup\limits_{z\neq 0}\frac{1}{\|z\|}\left[ \|y(nT;0,z,u_z)\|\|\psi\|+\| u_z\|_{L^2(0,nT; U)}\|B(\cdot)^*\varphi_n(\cdot;\psi)\|_{L^2(0,nT; U)}\right]\\
&\leq &
\sup\limits_{z\neq 0}\frac{\|y(nT;0,z,u_z)\|}{\|z\|}\|\psi\|+\sup\limits_{z\neq 0}\frac{\| u_z\|_{L^2(0,nT; U)}}{\|z\|}\|B(\cdot)^*\varphi_n(\cdot;\psi)\|_{L^2(0,nT; U)}.
\end{eqnarray*}
This, along with \eqref{estimate-state-control}, yields \eqref{1.17}. Hence, $(ii)$ is true.

\medskip

$(ii)\Longrightarrow(i)$. Arbitrarily fix $z\in Y$. Define a  space:
$$
\widetilde{Y}\triangleq\left\{\Bigr(\psi,\, B(\cdot)^*\varphi_n(\cdot;\psi)\Bigl)
\;\Bigm|\;
\psi\in Y\right\}\subset Y\times L^2(0, nT; U),
$$
with the norm:
\begin{equation}\label{norm1}
\left\|\Bigr(\psi,\, B(\cdot)^*\varphi_n(\cdot;\psi)\Bigl)\right\|_{\widetilde{Y}}=\delta\|\psi\|
+C\|B(\cdot)^*\varphi_n(\cdot;\psi)\|_{L^2(0,nT;U)}.
\end{equation}
We next define a functional ${\cal F}$ on this space by
\begin{equation}\label{1.18}
{\cal F}\Bigr(\psi,\, B(\cdot)^*\varphi_n(\cdot;\psi)\Bigl)=\langle z,\,\varphi_n(0;\psi)\rangle\qquad\text{for any }\psi\in Y.
\end{equation}
From \eqref{1.17}-\eqref{1.18}, we see that
$\|{\cal F}\|_{\widetilde{Y}^*}\leq \|z\|$.
Then by the Hahn-Banach theorem, there is a functional $\widetilde{\cal F}$ defined on $Y\times L^2(0, nT; U)$ so that
\begin{equation}\label{1.19}
\widetilde{\cal F}\Bigr(\psi,\, B(\cdot)^*\varphi_n(\cdot;\psi)\Bigl)={\cal F}\Bigr(\psi,\, B(\cdot)^*\varphi_n(\cdot;\psi)\Bigl)\qquad\text{for any }\psi\in Y
\end{equation}
and
\begin{equation}\label{1.20}
\left\vert\widetilde{\cal F}(\xi,\eta)\right\vert\leq\|z\|(\delta\|\xi\|+C\|\eta\|_{L^2(0,nT;U)})\qquad\text{for any }(\xi,\eta)\in Y\times L^2(0, nT; U).
\end{equation}

From the Riesz representation theorem, we can find $(y_z, u_z)\in Y\times L^2(0, nT; U)$ so that
\begin{equation}\label{1.21}
\widetilde{\cal F}(\xi,\eta)=\langle y_z,\,\xi\rangle+\int^{nT}_0\langle u_z(t),\, \eta(t)\rangle_U\mathrm dt
\qquad\text{for any }(\xi,\eta)\in Y\times L^2(0, nT; U).
\end{equation}
By taking $(\xi,\eta)=\Bigr(\psi,\, B(\cdot)^*\varphi_n(\cdot;\psi)\Bigl)$
 (with $\psi\in Y$ arbitrarily fixed) in \eqref{1.21}, using
\eqref{1.19} and \eqref{1.18}, we find
\begin{eqnarray*}\label{1.23}
\langle z,\,\varphi_n(0;\psi)\rangle=\langle y_z,\,\psi\rangle+\int^{nT}_0\langle u_z(t),\,  B(t)^*\varphi_n(t;\psi)\rangle_U\mathrm dt \qquad\text{for any }\psi\in Y.
\end{eqnarray*}
From this, as well as (\ref{contsyst}) (where $u=u_z$) and \eqref{adjoint-equation}, one can easily verify
\begin{equation}\label{1.24}
y_z=y(nT;0,z,-u_z).
\end{equation}
Meanwhile, it follows from
\eqref{1.20} and \eqref{1.21} that
$$
\|y_z\|\leq\delta\|z\|\qquad\text{and}\qquad \|-u_z\|_{L^2(0, nT; U)}\leq C\|z\|.
$$
These, together with \eqref{1.24}, give \eqref{estimate-state-control}. Hence, $(i)$ is true.

Thus, we end the proof of Lemma \ref{lemma2}.
 \end{proof}

The next  Lemma \ref{lemma3}  is on the connection between the periodic stabilization and the property $(i)$ in Lemma \ref{lemma2}.
\begin{lemma}\label{lemma3}
The following assertions  are equivalent:
\begin{enumerate}[(i)]
\item The system (\ref{contsyst}) is  periodically stabilizable.
\item Given $ \delta\in(0,1)$, there is $n_\delta\in\mathbb{N}^+$ and $ C_{\delta}>0$ so that for each $z\in Y$, there is $u_z\in L^2(0,n_\delta T;U)$ satisfying
\begin{equation}\label{1.25}
\|y(n_\delta T;0,z,u_z)\|\leq\delta\|z\|\qquad\text{and}\qquad
\|u_z\|_{ L^2(0,n_\delta T;U)}\leq C_\delta\|z\|.
\end{equation}
\item There is $ \delta\in(0,1)$, $n\in\mathbb{N}$, and $ C>0$ so that for each $z\in Y$, there is $u_z\in L^2(0,nT;U)$ satisfying
\begin{equation}\label{1.26}
\|y(n T;0,z,u_z)\|\leq\delta\|z\|\qquad\text{and}\qquad
\|u_z\|_{ L^2(0,n T;U)}\leq C\|z\|.
\end{equation}
\end{enumerate}
\end{lemma}

\begin{proof}
$(i)\Longrightarrow(ii)$. Since (\ref{contsyst}) is  periodically stabilizable, there is $T$-periodic
$K(\cdot)\in L^\infty(\mathbb{R}^+;{\cal L}(Y,U))$ satisfying \eqref{1.06}.
Note that  $\Phi_K(\cdot, 0)z$, with  $z\in Y$, is the solution to the  equation:
$$
\begin{cases}
 y'(t)=[A(t)+B(t)K(t)]y(t),\qquad t\geq0,\\
 y(0)=z.
 \end{cases}
 $$
Define, for each  $z\in Y$, a control $u_z: \mathbb{R}^+\rightarrow U$ by
\begin{eqnarray}\label{20,12-14}
u_z(t)=K(t)\Phi_K(t,0)z\;\;\mbox{for a.e.}\; t\in \mathbb{R}^+.
\end{eqnarray}
Then we have that  for each  $z\in Y$,
\begin{eqnarray}\label{25,12-7}
\Phi_K(t, 0)z=y(t;0,z;u_z)\;\;\mbox{for each}\;\;t\geq 0.
\end{eqnarray}
Arbitrarily fix $\delta\in(0,1)$. Let
\begin{equation*}
n_\delta=\left[\frac{1}{\omega}\ln \frac{M}{\delta}\right]+1,\qquad
C_\delta=\frac{M}{\sqrt{2\omega}}\|K\|_{L^\infty(\mathbb{R}^+;{\cal L}(Y,U))}
\end{equation*}
where $M$ and $\omega$ are given by \eqref{1.06}.
Then by \eqref{25,12-7}, \eqref{20,12-14} and \eqref{1.06}, after some direct computations, we get
 \eqref{1.25}.

$(ii)\Longrightarrow(iii)$. It is clear.

\medskip

 $(iii)\Longrightarrow(i)$.
 Suppose that $(iii)$ holds for some $\delta\in (0,1)$, $n\in \mathbb{N}^+$ and $C>0$. We claim that
 $W(\cdot)$, given by \eqref{1.15},  is finite valued.
  When this is done, we can apply  Lemma \ref{lemma1} to get $(i)$ of this lemma.

Now we prove the above claim. Arbitrarily fix $z\in Y$. Set $z_0\triangleq z$. By $(iii)$, we find
$u_1\triangleq u_{z_0}\in L^2(0,nT;U)$ satisfying \eqref{1.26} with $z=z_0$. Following this way, we can inductively get
     $\{u_k\}_{k=1}^\infty\subset L^2(0, nT;U)$  and $\{z_k\}_{k=0}^\infty\subset Y$ so that
     \begin{eqnarray}\label{27,12-8}
     \|z_k\|\leq \delta^k\|z_0\|\;\;\mbox{and}\;\;\|u_k\|_{L^2(0, nT;U)}\leq C\|z_{k-1}\|\leq C\delta^{k-1}\|z_0\|\;\;\mbox{for all}\; k\in \mathbb{N}^+;
     \end{eqnarray}
     and so that
      \begin{eqnarray}\label{28,12-8}
      z_k=y(nT;0,z_{k-1},u_{k})\quad \text{and}\quad u_{k+1}=u_{z_k}\quad\;\;\mbox{for all}\; k\in \mathbb{N}^+.
     \end{eqnarray}
Define a control $\tilde u_z: \mathbb{R}^+\rightarrow U$ by
 \begin{eqnarray}\label{29,12-8}
\tilde u_z(knT+t)=u_{k+1}(t), \; t\in [0,nT),\;\;\;k\in \mathbb{N}.
\end{eqnarray}
Then by \eqref{29,12-8} and \eqref{27,12-8}, we see
 \begin{eqnarray}\label{30,12-8}
\|\tilde u_z\|_{L^2(0,+\infty; U)}^2=\sum\limits_{k=0}^{+\infty}\|u_{k+1}\|_{L^2(0,nT; U)}^2<\infty.
\end{eqnarray}
 By \eqref{29,12-8}, \eqref{1.05}, \eqref{28,12-8},
\eqref{27,12-8} and \eqref{30,12-8}, we find
\begin{eqnarray}\label{31,12-8}
&&\|y(\cdot;0,z,\tilde u_z)\|_{L^2(0,+\infty; Y)}^2\nonumber\\
&=&
\sum\limits_{k=0}^{+\infty}\int^{nT}_0\left\|\Phi(t,0)z_k+\int^t_0\Phi(t,\tau)B(\tau)u_{k+1}(\tau)\mathrm d\tau\right\|^2\mathrm dt\nonumber\\
&\leq&
\sum\limits_{k=0}^{+\infty}\int^{nT}_02\left[\max\limits_{0\leq  t\leq nT}\|\Phi(t,0)\|_{\mathcal{L}(Y)}^2\|z_k\|^2+\int^t_0\|\Phi(t,\tau)B(\tau)\|_{\mathcal{L}(U,Y)}^2\mathrm d\tau\int^t_0\|u_{k+1}(\tau)\|_U^2\mathrm d\tau\right]\mathrm dt\nonumber\\
&\leq&
2nT\max\limits_{0\leq \tau\leq t\leq nT}\|\Phi(t,\tau)\|_{\mathcal{L}(Y)}^2
(1+nT\|B\|_{L^\infty}^2)\sum\limits_{k=0}^{+\infty}\left[\|z_k\|^2+\|u_{k+1}\|^2_{L^2(0, nT;U)}\right]\nonumber\\
&\leq&
2nT\max\limits_{0\leq \tau\leq t\leq nT}\|\Phi(t,\tau)\|_{\mathcal{L}(Y)}^2(1+nT\|B\|_{L^\infty}^2)(1+C^2)\sum\limits_{k=0}^{+\infty}\|z_k\|^2<\infty.
\end{eqnarray}

Now, from \eqref{1.13},  \eqref{30,12-8} and \eqref{31,12-8}, we see
that $J(\tilde u_z,z)<+\infty$ for all $z\in Y$. Thus, we have $W(z)<\infty$ for all $z\in Y$,
which leads to the desired claim.

 Thus, we end the proof.
\end{proof}

Now we begin to prove Theorem \ref{maintheorem1}.
\begin{proof}
"$(E_1)\Longrightarrow(E_2)$".
First of all, by the $T$-periodicity of $\Phi(\cdot,\cdot)$
and $B(\cdot)$ (see \eqref{1.05} and  $(H_2)$), we have that
\begin{eqnarray}\label{w32,12-8}
\Phi(t,s)^*=\Phi(r,s)^*\Phi(t,r)^*,\; \Phi(t,s)^*=\Phi(T+t,T+s)^*\;\;\mbox{for any}\; t\geq r\geq s\geq 0,
\end{eqnarray}
and that for any $n\in\mathbb{N}^+$ and any  $\ell,k\in \mathbb{N}^+$, with $k\geq \ell$,
\begin{eqnarray}\label{w33,12-8}
B(t)^*\Phi(\ell nT,t)^*=B((k-\ell)nT+t)^*\Phi(knT,(k-\ell)nT+t)^*\;\;\mbox{for any}\; t\in [0,nT].
\end{eqnarray}

Suppose that $(E_1)$ holds.
Let $\hat\delta\in (0,1)$.  CLAIM ONE: There is
$n_{\hat\delta}\in\mathbb{N}$ and $ C_{\hat\delta}>0$ so that \eqref{1.10} (where $\delta=\hat\delta$)
holds.

To show the above claim, we set
\begin{eqnarray}\label{32,12-8}
\delta\triangleq\hat\delta/\sqrt{2}\in(0,1)
\end{eqnarray}
By   Lemma \ref{lemma3} and $(E_1)$,
We have $(ii)$ of Lemma \ref{lemma3}, in particular, for $\delta$ given by \eqref{32,12-8}, we have
\begin{eqnarray}\label{35,12-9}
n_{\delta}\in\mathbb{N}\;\;\mbox{and}\;\; C_{\delta}>0
\end{eqnarray}
 so that
\eqref{1.25} holds. This, along with  Lemma \ref{lemma2}, yields that
\begin{equation}\label{36,12-9}
\|\varphi_{n_\delta}(0;\psi)\|\leq\delta\|\psi\|
+C_\delta\|B(\cdot)^*\varphi_{n_\delta}(\cdot;\psi)\|_{L^2(0,{n_\delta}T;U)}\qquad {\mbox{for any }}\psi\in Y,
\end{equation}
where $\delta$ is given by \eqref{32,12-8}, $n_\delta$ and $C_\delta$ are given by \eqref{35,12-9}.
By \eqref{36,12-9} and the Cauchy-Schwarz inequality,  we find
\begin{equation}\label{37,12-9}
\|\varphi_{n_\delta}(0;\psi)\|^2\leq 2 \delta ^2\|\psi\|^2
+2C_\delta^2\|B(\cdot)^*\varphi_{n_\delta}(\cdot;\psi)\|^2_{L^2(0,{n_\delta}T;U)}\qquad {\mbox{for any }}\psi\in Y.
\end{equation}
Arbitrarily fix $z\in Y$ and $k\in\mathbb{N}$. Taking  $\psi=\left[\Phi(n_\delta T,0)^*\right]^{\ell-1}z$,
with $\ell=1,\cdots,k$, in \eqref{37,12-9}, then multiplying the both sides by $(2\delta^2)^{k-\ell}$, using
 \eqref{w32,12-8}, we obtain that when $\ell=1,\cdots,k$,
\begin{eqnarray}\label{1.32}
&&(2\delta^2)^{k-\ell}\|\varphi_{\ell n_\delta}(0;z)\|^2\nonumber\\
&=&
(2\delta^2)^{k-\ell} \left\|\varphi_{n_\delta}(0;\left[\Phi(n_\delta T,0)^*\right]^{\ell-1}z)\right\|^2\nonumber\\
&\leq&
(2\delta^2)^{k-\ell+1}\left\|\left[\Phi(n_\delta T,0)^*\right]^{\ell-1}z\right\|^2
+(2\delta^2)^{k-\ell}2C_\delta^2\|B(\cdot)^*\varphi_{n_\delta}(\cdot;\,[\Phi(n_\delta T,0)^*]^{\ell-1}z)\|^2_{L^2(0,{n_\delta}T;U)}\nonumber\\
&\leq&
(2\delta^2)^{k-\ell+1}\|\varphi_{(\ell-1) n_\delta}(0;z)\|^2+
2C_\delta^2\|B(\cdot)^*\Phi(\ell n_\delta T, t)^*z\|_{L^2(0,{n_\delta}T;U)}^2.
\end{eqnarray}
(In the last inequality of \eqref{1.32}, we used the fact: $2\delta^2<1$.)
Meanwhile, by \eqref{w32,12-8} and \eqref{w33,12-8}, we see that  when $\ell=1,\cdots,k$,
\begin{eqnarray}\label{39,12-9}
\int^{n_\delta T}_0\|B(t)^*\Phi(\ell n_\delta T, t)^*z\|_U^2\mathrm dt
&=&\int^{(k-\ell +1)n_\delta T}_{(k-\ell)n_\delta T}\|B(\hat t)^*\Phi(k n_\delta T,\hat t)^*z\|_U^2\mathrm d{\hat t}\nonumber\\
&=& \|B(\cdot)^*\varphi_{kn_\delta}(\cdot;z)\|^2_{L^2((k-\ell)n_\delta T,(k-\ell+1){n_\delta}T;U)}.
\end{eqnarray}
Thus, by \eqref{39,12-9} and  \eqref{1.32},  we find that  when $\ell=1,\cdots,k$,
\begin{eqnarray}\label{1.32-1}
&&(2\delta^2)^{k-\ell}\|\varphi_{\ell n_\delta}(0;z)\|^2\nonumber\\
&\leq&
(2\delta^2)^{k-\ell+1}\|\varphi_{(\ell-1) n_\delta}(0;z)\|^2+2C^2_\delta
 \|B(\cdot)^*\varphi_{kn_\delta}(\cdot;z)\|^2_{L^2((k-\ell)n_\delta T,(k-\ell+1){n_\delta}T;U)}.
\end{eqnarray}
Taking the sum from $\ell=1$ to  $\ell=k$ in \eqref{1.32-1}, we get
\begin{eqnarray}\label{1.33}
\|\varphi_{kn_\delta}(0;\psi)\|^2
&\leq& (2\delta^2)^{k}\|z\|^2+
2C_\delta^2\|B(\cdot)^*\varphi_{kn_\delta}(\cdot;z)\|^2_{L^2(0,k{n_\delta}T;U)}\nonumber\\
&\leq&
 \left[(\sqrt{2}\delta)^{k}\|z\|+
\sqrt{2}C_\delta\|B(\cdot)^*\varphi_{kn_\delta}(\cdot;z)\|^2_{L^2(0,k{n_\delta}T;U)}\right]^2.
\end{eqnarray}
Now, by \eqref{1.33} and by letting $n_{\hat\delta}=n_\delta$ and $C_{\hat\delta}=\sqrt{2}C_\delta$
(where $n_\delta$ and $C_\delta$ are given by \eqref{35,12-9}), we get CLAIM ONE.
Hence, $(E_2)$ is true.

 $(E_2)\Longrightarrow  (E_3)$.   It is trivial.

   $(E_3)\Longrightarrow  (E_1)$.
  According to Lemma
\ref{lemma2},
    $(E_3)$ is equivalent  to $(iii)$ of Lemma \ref{lemma3}. Then by  Lemma \ref{lemma3}, we get $(E_1)$.

    In summary,   we complete the proof of Theorem \ref{maintheorem1}.
\end{proof}

\section{Further discussions}
\subsection{On the detectability inequality} \label{subsection3.1}

In this subsection, we will explain why we call \eqref{1.10} (or \eqref{1.11}) as a detectability inequality.

The concept of the detectability arises from the finite-dimensional  system.
Let $A\in \mathbb{R}^{n\times n}$ and $B\in \mathbb{R}^{n\times m}$. Consider the system:
\begin{eqnarray}\label{40,12-12}
z'(t)=A^*z(t);\; w(t)=B^*z(t),\;\;t\geq 0.
\end{eqnarray}
We quote the concept of detectability from  \cite{Hautus}: {\it The system
\eqref{40,12-12} is detectable, if
\begin{equation}\label{1.12-6}
w(\cdot)=0\; \mbox{over}\; [0,\infty) \Longrightarrow \lim\limits_{t\rightarrow+\infty}z(t)=0.
\end{equation}}
For the pair of matrices $[A,B]$, the following statements are equivalent: (See  \cite{Hautus} or \cite{Sontag}.)
\begin{itemize}
\item[$(a_1)$] The system $y'(t) = Ay(t)+Bu(t),\; t\geq 0$ is stabilizable.
\item [$(a_2)$] The system \eqref{1.12-6} is detectable.
\item [$(a_3)$] There is $L\in \mathbb{R}^{n\times m}$ so that
$A^*+LB^*$ is stable.
\end{itemize}
With the aid of  the above equivalence, the concept of the detectability for
infinite-dimensional  time-invariant systems was extended in \cite{Weiss}.
Motivated by the extension in \cite{Weiss}, we  define  the detectability for
$[A(\cdot),B(\cdot)]$ in the infinitely-dimensional periodic setting
 $(H_1)$-$(H_3)$ in the following manner:

{\it The system
\begin{eqnarray}\label{42,12-12}
z'(t)=A(t)^*z(t);\;\;  w(t)=B(t)^*z(t),\;\;t\geq 0
\end{eqnarray}
is detectable, if there is a
$T$-periodic operator-valued function
$L(\cdot)\in L^\infty(\mathbb{R}^+;{\cal L}(Y,U))$ so that the system
\begin{equation*}\label{1.12-4}
z'(t) = \left[A(t)^*+L(t)B(t)^*\right]z(t),\qquad t\geq0,
\end{equation*}
 is exponentially stable.}

 From the above definition and Theorem \ref{maintheorem1}, we can easily verify what follows:

 \begin{itemize}
 \item[$(b_1)$] The system \eqref{contsyst}
 is periodically stabilizable if and only if the system \eqref{42,12-12} is detectable.
 \end{itemize}
 (Here, we used the fact: $\varphi_n(nT-\cdot,\psi)$  solves the first equation of \eqref{42,12-12}, over $[0,nT]$, with the initial condition $z(0)=\psi$.)
  Thus, the inequality \eqref{1.10} (or \eqref{1.11}) is equivalent to the detectability
 of \eqref{42,12-12}. {\it This is why we call \eqref{1.10} (or \eqref{1.11}) as a detectability inequality.}

For $[A(\cdot),B(\cdot)]$ in the infinitely-dimensional periodic setting  $(H_1)$-$(H_3)$,
it deserves mentioning the following connection between "detectability of \eqref{42,12-12}" and \eqref{1.12-6} (where $(z,w)$ solves \eqref{42,12-12}):

\begin{itemize}
\item [$(c_1)$] We have that "detectability of \eqref{42,12-12}" $\Longrightarrow$ \eqref{1.12-6}.  In fact, by $(a_3)$ and $(b_1)$, we see that
 the system \eqref{contsyst}
 is periodically stabilizable. Then  given  $\delta\in(0,1)$, we can use Theorem 1 to find $n_\delta\in\mathbb{N}^+$ and $ C_{\delta}>0$ so that
\eqref{1.10} holds for any  $k\in\mathbb{N}$.
Meanwhile, arbitrarily fix $k\in \mathbb{N}^+$ and $\psi\in Y$. Let
$$
z(t;k,\psi)=\varphi_{kn_\delta}(kn_\delta T-t;\psi),\; \; t\in [0, kn_\delta T].
$$
Then $z(\cdot;k,\psi)$ solves the first equation of \eqref{42,12-12}, over $[0,kn_\delta T]$, with the initial condition $z(0)=\psi$.
This, along with \eqref{1.10}, yields
\begin{eqnarray*}\label{1.12-5}
\|z(kn_\delta T)\|\le\delta^k\|z(0)\|
+C_\delta\|w(\cdot)\|_{L^2(0,kn_\delta T;U)},
\end{eqnarray*}
which leads to
$$
\lim\limits_{k\rightarrow+\infty}z(kn_\delta T)=0, \;\;\mbox{when}\;\; w(\cdot)=0.
$$
So \eqref{1.12-6} is true.
\item [$(c_2)$] It is not true that \eqref{1.12-6} $\Longrightarrow$  "detectability of \eqref{42,12-12}".
Here is an counterexample.
Let $B(\cdot)=0$ and $D(\cdot)=0$. Take an operator $A$ so that the system $\dot y=Ay$ is polynomially  stable  but not exponentially stable. Then we clearly have  \eqref{1.12-6}, but have no the stabilization.
Then by $(b_1)$,  the system \eqref{42,12-12} is not detectable in this case.
\end{itemize}

\subsection{Connection of  detectability inequality and  unique continuation}

In this subsection, besides $(H_1)$-$(H_3)$,   we further assume
\begin{enumerate}[($H_4$)]
\item The operator $A$ is compact.
\end{enumerate}
We will present an equivalence between the detectability inequality \eqref{1.10} (or \eqref{1.11}) and
a qualitative unique continuation property for the adjoint system \eqref{adjoint-equation}, under the setting
$(H_1)$-$(H_4)$. This qualitative unique continuation was introduced in \cite{WangXu} (see also \cite{WangXu-book}).

We start with introducing some notation:
We write $Y^C$ for the complexification of $Y$. For each $L\in{\cal L}(Y)$, we denote by
$L^C$ the complexification of $L$.
We write $\mathbb{B}$ for the open unit ball in $\mathbb{C}^1$ and $\mathbb{B}(0,\delta)$  for the open ball in $\mathbb{C}^1$, centered at the origin and of radius $\delta>0$.

We next introduce some concepts.
\begin{enumerate}[(I)]
\item  {\it The Poincar\'{e} map.} ~
We recall the following Poincar\'{e} map (see  Page 197, \cite{Henry}):
\begin{equation*}\label{poincare}
{\cal P}(t) \triangleq\Phi(t+T,t),\;\;t\in\mathbb{R}^+.
\end{equation*}
We have
\begin{equation*}\label{WGS1.8}
\sigma({\cal P}(t)^C)\setminus\{0\}=\{\lambda_j\}^\infty_{j=1}\;\;\mbox{for each}\;\; t\geq 0,
\end{equation*}
where $\lambda_j$, $j=1,2,\dots$, are all distinct non-zero eigenvalues of the compact operator ${\cal P}(0)^C$ so that $\lim_{j\rightarrow\infty}|\lambda_j|=0$. Thus, there is a unique $n\in \mathbb{N}^+$ so that
 \begin{equation*}\label{wgs1.9}
 |\lambda_j|\ge1,\quad j\in\{1,2,\cdots, n\}\quad~{\mbox{and}}~\quad |\lambda_j|<1,\quad j\in\{n+1,n+2,\cdots \}.
 \end{equation*}
Set
 \begin{equation}\label{decay}
 \bar \delta\triangleq \max\{|\lambda_j|\; |\;  j>n\}<1.
 \end{equation}
 Let $l_j$ be the algebraic multiplicity of $\lambda_j$ for each $j\in \mathbb{N}^+$, and write
 \begin{equation}\label{number}
n_0\triangleq l_1+ \cdots +l_n.
 \end{equation}
 \item
 {\it The Kato projection.} ~ Arbitrarily fix  $\delta\in (\bar \delta, 1)$, where $\bar \delta$ is given by (\ref{decay}).   Let
$\Gamma$ be the circle $\partial \mathbb{B}\left(0, \delta\right)$ with the anticlockwise direction in $\mathbb{C}^1$.
The  Kato projection is given by (see \cite{Kato})
\begin{equation*}\label{wgs1.11}
{\cal K}(t) =\displaystyle\frac{1}{ 2\pi i}\int_\Gamma (\lambda \mathrm{id}-{\cal P}(t)^C)^{-1}d\lambda,\;\; t\geq 0.
\end{equation*}
Let
\begin{eqnarray*}
P(t)\triangleq \bigr(I_Y-{\cal K}(t)\bigl)\bigm|_Y\;\;(\mbox{the restriction of}\;(\mathrm{I_Y}-{\cal K}(t))\;\mbox{on}\; Y),\;\; t\geq 0,
\end{eqnarray*}
where $I_Y$ is the identity operator on $Y$. Let
\begin{eqnarray}\label{wgs1.15}
P\triangleq P(0)\;\;\mbox{and}\;\;Y_{u}\triangleq P(0)Y.
\end{eqnarray}

  \end{enumerate}

\begin{theorem}\label{theorem2,12-13}
Suppose that $(H_1)$-$(H_4)$ hold. Let $P$ and $Y_{u}$  be given by (\ref{wgs1.15}). Let  $n_0$  be given by   (\ref{number}).
Then the following statements are equivalent:

\noindent $(i)$ There is $ \delta\in(0,1)$,  $n\in\mathbb{N}^+$  and $C>0$ so that the  detectability inequality \eqref{1.10} holds.

\noindent $(ii)$  If  $\xi\in P^*Y_{u}$ and $B^*(\cdot)\Phi(n_0T, \cdot)^*\xi=0$ over $(0,n_0T)$, then $\xi=0$.

\end{theorem}
\begin{proof}  According to Theorem \ref{maintheorem1}, the periodic stabilization of \eqref{contsyst} is equivalent to $(i)$ in Theorem \ref {theorem2,12-13}. Meanwhile, according to Theorem 2.1 in \cite{WangXu-book},
the periodic stabilization of \eqref{contsyst} is equivalent to $(ii)$ in Theorem \ref {theorem2,12-13}.
  Hence, the statements $(i)$ and $(ii)$ in Theorem \ref {theorem2,12-13}
 are equivalent. This ends the proof.
\end{proof}

\subsection{Example}
We will give a time periodic controlled heat equation which can be put into our framework $(H_1)$-$(H_3)$,
and which is not null controllable and not stable with the null control, but satisfies the  detectability inequality,
consequently, is periodically stabilizable (by Theorem \ref{maintheorem1}).

\begin{example}
Consider the following heat equation
\begin{equation}\label{48,12-13}
\begin{cases}
y_t(x,t)-\left(\triangle+3\sin^2t\right ) y(x,t)=u(t)\sin x,\quad& (x,t)\in (0,\pi)\times(0,\infty),\\
y(0,t)=y(\pi,t)=0,& t\in (0,\infty),\\
y(\cdot,0)\in L^2(0,\pi).
\end{cases}
\end{equation}
The equation \eqref{48,12-13} can be put into the setting $(H_1)$-$(H_3)$ in the following manner:
Let  $Y=L^2(0,\pi)$; $U=\mathbb{R}$ and $T=\pi$. Let $A=\triangle$, with $D(A)=H^2(0,\pi)\bigcap H^1_0(0,\pi)$;   $D(t)=(3\sin^2t)I_Y$ for each $t\geq 0$; $B(t)\equiv B$ for all $t\geq 0$, where $B: U\rightarrow Y$ is defined by $B(\alpha)=\alpha \sin x$, $\alpha\in U$.

First, we show that \eqref{48,12-13}, with $u=0$, is not stable. To this end,
we let $y$ be the solution to \eqref{48,12-13}, where $u=0$ and $y(x,0)=\sin x$, $x\in (0,\pi)$.
Write
\begin{eqnarray}\label{wang44,12-14}
a(t)\triangleq\int_0^\pi y(x,t)\sin xdx,\;\;t\geq 0.
\end{eqnarray}
Then by \eqref{48,12-13} and \eqref{wang44,12-14}, we find that
\begin{eqnarray}\label{49,12-13}
\dot a(t)=(-1+3\sin^2t)a(t),\;\;t\geq 0;\;\;\;\; a(0)>0.
\end{eqnarray}
Since
\begin{eqnarray*}
\int_0^t(-1+3\sin^2s)ds=\frac{t}{2}-\frac{3}{4}\sin(2t),\; t\geq 0,
\end{eqnarray*}
it follows from \eqref{49,12-13} that
$$
a(t)=e^{\frac{t}{2}-\frac{3}{4}\sin(2t)}a(0),\;\;t\geq 0;\;\;\;\;a(0)>0.
$$
These yields that that $\lim_{t\rightarrow\infty}a(t)=\infty$.
From this, we see  that
\eqref{48,12-13}, where $u=0$, is unstable.

 Next, we  show that  \eqref{48,12-13} is not null controllable.
For this purpose, we arbitrarily fix a control $u$. Let $y(x,t;u)$ be the solution of \eqref{48,12-13}
with the aforementioned $u$ and with  $y(x,0)=\sin (2x)$ ($x\in (0,\pi)$).
 Write
$$
z(t;u)\triangleq\int^\pi_0y(x,t;u)\sin(2x)\mathrm dx,\quad t\geq 0.
$$
Then by \eqref{48,12-13}, we have
\begin{eqnarray}\label{50,12-13}
\begin{cases}
\dot z(t;u)=(3\sin^2t-4)z(t;u),\;\;t\geq 0,\\
z(0;u)={\pi}/{2}.
\end{cases}
\end{eqnarray}
From \eqref{50,12-13}, we see that
$z(t;u)\neq 0$ for any $t\geq 0$. Thus, for any $u$ and any $t\geq 0$, we have $y(\cdot,t;u)\neq 0$. So
 \eqref{48,12-13} is not null controllable.

Now we show the detectability inequality \eqref{1.11} holds for this example.
To this end, we take
\begin{eqnarray}\label{51,12-13}
\delta=e^{-\pi};\qquad  n=1;\qquad C=2e^{2\pi}.
\end{eqnarray}
Since $T=\pi$ in this example, we see from \eqref{51,12-13} that $nT=\pi$. So the adjoint equation
\eqref{adjoint-equation} in the current case reads as:
\begin{equation}\label{1.70}
\varphi_t(x,t)+\left(\triangle+3\sin^2t)\right ) \varphi(x,t)=0\quad (x,t)\in(0,\pi)\times [0,\pi];
\qquad \varphi(x,\pi)=\psi,
\end{equation}
where $\psi \in L^2(0,\pi)$.
Write $\varphi(x,t;\psi)$ ($x\in (0,\pi)$, $t\in [0,\pi]$) for the solution to \eqref{1.70}. Then it is exactly
the solution $\varphi_n(\cdot;\psi)$ to \eqref{adjoint-equation} with $n=1$.
Let
\begin{equation}\label{53,12-13}
\varphi(x,t;\psi)=\sum\limits_{k=1}^{\infty}a_k(t)\sin kx,\; x\in (0,\pi), t\in [0,\pi].
\end{equation}
By \eqref{53,12-13} and \eqref{1.70}, we have that for each $k\in \mathbb{N}^+$,
$$
\dot a_k(t)=(k^2-3\sin^2t)a_k(t),\;\;t\in [0,\pi].
$$
From the above, one has
\begin{equation}\label{54,12-13}
|a_k(0)|\leq e^{-(k^2-3)\pi}|a_k(\pi)|,\;\;k\in\mathbb{N};\qquad\;\; |a_1(t)|\geq e^{t-\pi} |a_1(\pi)|,\;\; t\in[0,\pi].
\end{equation}
By \eqref{54,12-13} and \eqref{53,12-13}, we can easily check that
\begin{eqnarray}\label{55,12-13}
\|\varphi_n(0;\psi)\|^2
&=&\int_0^\pi\varphi^2(x,0;\psi)dx \nonumber\\
&=&\frac{\pi}{2}
\sum\limits_{k=1}^\infty a_k(0)^2\nonumber\\
&\leq&\frac{\pi}{2}\left[
e^{4\pi}a_1(\pi)^2+\sum\limits_{k=2}^\infty e^{-2(k^2-3)\pi}a_k(\pi)^2\right]\nonumber\\
&\leq&\frac{\pi}{2}\left[
e^{4\pi}a_1(\pi)^2+\sum\limits_{k=1}^\infty e^{-2\pi}a_k(\pi)^2\right]\nonumber\\
&=&\frac{\pi}{2}
e^{4\pi}a_1(\pi)^2+e^{-2\pi}\|\psi\|^2
\end{eqnarray}
and
\begin{eqnarray}\label{56,12-13}
\|B(\cdot)^*\varphi_n(\cdot;\psi)\|_{L^2(0,\pi)}^2
&=&
\int^\pi_0a_1(t)^2
\left(\frac{\pi}{2}\right)^2\mathrm dt\nonumber\\
&\geq &
\frac{\pi^2}{8}(1-e^{-2\pi})a_1(\pi)^2\nonumber\\
&\geq &
\frac{\pi}{8}a_1(\pi)^2.
\end{eqnarray}
From \eqref{55,12-13} and  \eqref{56,12-13}, we  get \eqref{1.11}.

Finally, by  Theorem \ref{maintheorem1} and  \eqref{1.11}, we see that
 the system \eqref{48,12-13} is periodically stabilizable.
\end{example}

\end{document}